\documentclass[12pt,a4paper,reqno]{amsart}
\usepackage{mathrsfs}
\usepackage{calc}
\usepackage{cite}
\usepackage{amssymb,amsmath,latexsym, amsthm,enumerate,amsfonts,graphicx}
\usepackage[mathscr]{euscript}

\input amssym.def
\input amssym

\def\dj{d\kern-0.4em\char"16\kern-0.1em}
\def\Dj{\mbox{\raise0.3ex\hbox{-}\kern-0.4em D}}

\def\be{\begin{equation}}
\def\ee{\end{equation}}
\def\bena{\begin{eqnarray*}}
\def\ena{\end{eqnarray*}}

\def\t{\tau}
\def\s{\sigma}

\def\suml{\sum\limits}

\def\dss{\displaystyle}

\newcommand{\WF}{\operatorname{WF}}
\newcommand{\sing}{\operatorname{singsupp}}
\newcommand{\card}{\operatorname{card}}

 \def\D{\mathcal{D}}
 \def\E{\mathcal{E}}
 \def\Rd{\mathbf{R}^d}
 \def\Z{\mathbf{Z}_+}

\def\N{\mathbf{N}}

\def\lf {\lfloor}
\def\rf{\rfloor}
\newcommand{\supp}{\operatorname{supp}}
\numberwithin{equation}{section}

\newtheorem{te}{Theorem}[section]
\newtheorem{lema}{Lemma}[section]
\newtheorem{prop}{Proposition}[section]
\newtheorem{cor}{Corollary}[section]

\theoremstyle{definition}
\newtheorem{ex}{Example}[section]
\newtheorem{de}{Definition}[section]
\theoremstyle{remark}
\newtheorem{rem}{Remark}[section]

\title{\textbf{Superposition and propagation of singularities for extended Gevrey regularity}}

\author{Stevan Pilipovi\' c}

\address{Department of Mathematics and Informatics,
University of Novi Sad, Novi Sad, Serbia}

\email{stevan.pilipovic@dmi.uns.ac.rs}

\author{Nenad Teofanov}

\address{Department of Mathematics and Informatics,
University of Novi Sad, Novi Sad, Serbia}

\email{nenad.teofanov@dmi.uns.ac.rs}

\author{Filip Tomi\'c}

\address{Faculty of Technical Sciences,
University of Novi Sad, Novi Sad, Serbia}

\email{filip.tomic@uns.ac.rs}



\begin{document}
\begin{abstract}
We use sequences which depend on two parameters
to define families of ultradifferentiable functions which contain
Gevrey classes. It is shown that such families
are closed under superposition, and therefore inverse closed as well.
Furthermore, we study partial differential operators
whose coefficients satisfy the extended Gevrey regularity.
To that aim we introduce appropriate wave front sets and derive a theorem on propagation of singularities.
This extends related known results  in the sense that weaker assumptions on the regularity of
the coefficients are imposed.
\end{abstract}

\maketitle

\par

\section{Introduction}

Gevrey classes serve as an important reservoir  of functions in the context
of different aspects of general theory
of linear partial differential operators such as hypoellipticity,
local solvability and propagation of singularities, since they
describe regularities stronger than smoothness and weaker than analyticity \cite{ABM, Herm, Menon}.
For example, the Cauchy problem for weakly hyperbolic linear partial differential
equations (PDEs) is well-posed for certain values of the Gevrey index $t$, while
it is ill-posed in the class of analytic functions, cf. \cite{Chen, Rodino} and the references given there.

\par

Since the union of Gevrey classes is strictly contained in the class of
smooth functions, it is of interest to study intermediate spaces of smooth functions
by introducing appropriate regularity conditions. This is done in \cite{PTT-02} by observing
two-parameter dependent sequences of the form $ \{ p^{\t p^\s} \}_{p \in {\mathbf N}} $,
$ \t > 0 $, $ \s > 1$, instead of the Gevrey sequence $ \{ p!^t \}_{p \in {\mathbf N}} $,
$ t>1$.
The corresponding  families of ultradifferentiable functions, denoted by $ \E_{\t,\s}(U)$, extend Gevrey regularity,
see Section \ref{sec01} for the precise definition. We refer to  \cite{PTT-02, TT} for the main properties of such spaces,
and note that they can be used e.g. in situations when hypoellipticity of a PDE is
better than $C^\infty$ but worse than Gevrey hypoellipticity.
In particular, the space $ \E_{\{1,2\}}(U)$ is recently explicitly used in the study of  strictly hyperbolic equations
to capture the regularity of the coefficients in the space variable (with low regularity in time),
which ensures that the corresponding Cauchy problem is well posed in appropriate solution spaces.
We refer to \cite{CL} for details.

\par

In this paper we give a further insight to the extended Gevrey regularity by
proving the  superposition theorem for $ \E_{\t,\s }(U)$, Theorem \ref{TeoremaKompozicija},
which immediately implies the inverse closedness property.
In the proof we employ a generalized version of Fa\'a di Bruno formula and modified version of \emph{Fa\'a di Bruno property} of the sequences $\dss \Big\{\frac{M_p^{\t,\s}}{p!}\Big\}_{p \in {\mathbf N}}$ (Lemma \ref{modifiedFaa}), following the ideas presented in \cite{RS}.

\par

Another goal of this paper is to derive
propagation of singularities when the coefficients $a_{\alpha}(x)$ of the partial differential operator
$ P(x,D)=\sum_{|\alpha|\leq m}a_{\alpha}(x)D^{\alpha}$ belong to $ \E_{\t,\s}(U)$,
see Theorem \ref{FundamentalTheoremNK}.
Note that analytic coefficients were treated in \cite[Theorem 8.6.1]{HermanderKnjiga},
while \cite[Theorem 1.1]{PTT-02} treats constant coefficients.

It turns out that an additional information is needed in the study of
operators with variable coefficients, since it is not possible to use commutativity properties which hold
true when the coefficients are constants.
The main tools to overcome these difficulties are the inverse closedness property
and careful study of summands in generalized Fa\'a-di Bruno's formula,
which gives rise to an explicit construction of approximate solution in Subsections \ref{subsecrepresNK}.
Apart from this we use a new result in microlocal analysis,
Theorem \ref{TeoremaKarakterizacija} which shows that instead of admissible sequences of cut-off functions used in \cite{PTT-02},
a single cut-off function can be used in the definition of wave-front set  $\WF_{\t,\s}(u)$, $u \in \D'(U)$.
We refer to \cite{PTT-02} for a discussion on different types of wave-front sets in the context of ultradifferentiable functions.

\par

We summarize the paper as follows. In Section \ref{sec01} we discuss regularity conditions related to the sequences of the form
$M_p^{\t,\s}=p^{\t p^{\s}}$, $\t>0$, $\s>1$, $p\in \N $ (cf. \cite{PTT-01,PTT-02,TT}),
and introduce the spaces of ultradifferentiable functions $ \E_{\t,\s}(U)$.
In Section \ref{ChapterWaveFront} we introduce wave front sets $\WF_{\t,\s}(u)$, $u \in \D'(U)$,
in the context of extended Gevrey regularity and explain
enumeration, an important technical tool in our analysis. The main result there is Theorem \ref{TeoremaKarakterizacija}
which offers an equivalent definition of $\WF_{\t,\s}(u)$ to be used further on.
Finally, in Section \ref{SectionNonConstant} we prove the propagation of singularities, Theorem \ref{FundamentalTheoremNK}.
The proof is given in details since it contains new nontrivial observations and facts in comparison with the proof of \cite[Theorem 1.1]{PTT-02}.

\par

\subsection{Notation}
Throughout the paper we use the standard notation for  sets of numbers and spaces of distributions, e.g.
${\bf N}$,$\Z$, ${\bf R}_+$ denote the sets of nonnegative
integers, positive integers, and positive real
numbers, respectively, and Lebesgue spaces over an open set $ \Omega \subset \Rd$ a
re denoted by $ L^p (\Omega) $, $ 1\leq p <\infty$. For $x \in \Rd$
we put  $\langle x \rangle=(1+|x|^2)^{1/2}$.
The integer parts (the floor and the ceiling functions) of  $x\in {\bf R}_+$
are denoted by $\lf x \rf:=\max\{m\in \N\,:\,m\leq x\}$ and $\lceil x \rceil:= \min\{m\in \N\,:\,m\geq x\}$. For a multi-index
$\alpha=(\alpha_1,\dots,\alpha_d)\in {\bf N}^d$ we write
$\partial^{\alpha}=\partial^{\alpha_1}\dots\partial^{\alpha_d}$, $\dss D^{\alpha}= (-i )^{|\alpha|}\partial^\alpha$, and
$|\alpha|=|\alpha_1|+\dots |\alpha_d|$.
Open ball of radius $r>0$ centered at $x_0\in\Rd$ is denoted by $B_r(x_0)$, and $\card A$ denotes the cardinal number of $A$.
The Fourier transform of $u \in L^1 (\Rd) $ is normalized as
$$\dss {\mathcal F }_{x\to \xi}u(x)=\widehat u(\xi)=\int_{\Rd} u(x)e^{-2\pi i \langle x, \xi \rangle}\,dx =\int_{\Rd} u(x)e^{-2\pi i x \xi}\,dx, \;\;\; \xi \in \Rd,$$
and the convolution of $f,g \in L^1 (\Rd) $ is given by $f*g(x)=\int_{\Rd}f(x-y)g(y) dy$.
Both transforms can be extended in different ways.

By $C^{\infty}(K)$
we denote the set of smooth functions
on a regular compact set $K$, and
$\D(U)$ and $\E(U)$ denote test function spaces for
the space of Schwartz distributions $\D'(U)$ , and for the space of compactly supported distributions $\E'(U)$,
respectively.

\par

We will use the Stirling formula: $N!=N^N e^{-N}\sqrt{2\pi N}e^{\theta_N \over 12N}$,
for some $0<\theta_N<1$, $N\in\Z$, and formulas for multinomial coefficients:
\begin{multline}
{|a|\choose a_1,a_2,\dots a_m}:={|a|\choose a_1}{|a|-a_1\choose a_2}\dots{|a|-a_1-\dots-a_{m-2}\choose a_{m-1}}\nonumber\\
=\frac{|a|!}{a_1!a_2!\dots a_m!}
= \sum_{k=1}^m {|a|-1\choose a_1,..., a_k-1,... a_m}\nonumber,
 \label{PascalTriangle}
\end{multline}
where $|a|=a_1+a_2+\dots+a_m$, $a_k\in \N$, $k\leq m$.

\section{Classes ${\E}_{\t, \s}(U)$ and superposition property}
\label{sec01}

In this section we introduce test function spaces denoted by  ${\E}_{\t, \s}(U)$
via defining sequences of the form
$M_p^{\t,\s}=p^{\t p^{\s}}$, $p\in \N $, depending on parameters $\t>0$ and  $\s>1$.
The flexibility obtained by introducing the two-parameter dependence
enables the study of smooth functions which are less regular than the Gevrey functions.
When $\t>1$ and  $\s=1$ we  recapture the Gevrey classes.

The spaces  ${\E}_{\t, \s}(U)$ are already studied in \cite{PTT-01, PTT-02}.
Here we recall their basic properties which are used in the rest of the paper, and collect new results in Subsection
\ref{subsecInverse}.
We employ Komatsu's approach \cite{Komatsuultra1} to spaces of
ultradifferentiable functions. Another widely used approach is that of Braun, Meise, Taylor, Vogt and their
collaborators, see e.g. \cite{BMT} and the recent contribution \cite{RS}.
These two approaches are equivalent in many interesting situations, cf. \cite{LO} for more details.

Essential properties of the defining sequences are
given in the following lemma. We refer to \cite{PTT-01} for the proof.

\begin{lema}
\label{osobineM_p_s}
Let $\tau>0$, $\s>1$ and $M_p^{\tau,\s}=p^{\tau p^{\s}}$, $p\in \Z$, $M_0^{\tau,\s}=1$.
Then there exists an  increasing sequence of positive numbers $C_q$, $q\in \N$, and a constant $C>0$ such that:
\vspace{2mm}\\
\vspace{1mm}
$(M.1)$ $(M_p^{\t,\s})^2\leq M_{p-1}^{\t,\s}M_{p+1}^{\t,\s}$, $p\in \Z$\\
\vspace{1mm}
$\overline{(M.2)}$ $M_{p+q}^{\t,\s}\leq C^{p^{\s}+q^{\s}}M_p^{\t 2^{\s-1},\s}M_q^{\t 2^{\s-1},\s}$, $p,q\in \N$,\\
\vspace{1mm}
$\overline{(M.2)'}$ $M_{p+q}^{\t,\s}\leq C_q^{p^{\s}}M_p^{\t,\s}$, $p,q\in \N$,\\
\vspace{1mm}
$(M.3)'$
$ \displaystyle
\suml_{p=1}^{\infty}\frac{M_{p-1}^{\t,\s}}{M_p^{\t,\s}} <\infty.
$

Moreover, there exist constants $A,B,C>0$ such that
$$
M_p^{\tau,\s}\leq A C^{p^{\s}}{\lf p^{\s}
\rf}!^{\frac{\tau}{\s}}\quad and \quad {\lf p^{\s}  \rf}!^{\frac{\tau}{\s}}\leq B
M_p^{\tau,\s}.
$$
\end{lema}

Note that
$\dss \frac{M_{p-1}^{\t,\s}}{M_p^{\t,\s}}\leq \frac{1}{(2p)^{\tau (p-1)^{\s-1}}}$.
and
$$
\lf p^{\s} \rf!^{\frac{\t}{\s}}\sim  (2\pi)^{\frac{\t}{2\s}}p^{\frac{\t}{2}}e^{-\frac{\t p^{\s}}{\s}}M_p^{\t,\s} , \quad p\to \infty.
$$

For any given values $\tau,h>0$, $\s>1$ and a regular compact set $K\subset \Rd$, we denote by ${\E}_{\t, {\s},h}(K)$
the Banach space of  functions $\phi \in  C^{\infty}(K)$ such that
\begin{equation} \label{Norma}
\| \phi \|_{{\E}_{\t, {\s},h}(K)}=\sup_{\alpha \in \N^d}\sup_{x\in K}
\frac{|\partial^{\alpha} \phi (x)|}{h^{|\alpha|^{\s}}  M_{|\alpha|} ^{\t,\s} }<\infty.\,
\end{equation}
Obviously,
$$ \displaystyle
{\E}_{\t_1, {\s_1},h_1}(K)\hookrightarrow {\E}_{\t_2,
{\s_2},h_2}(K), \;\;\;
0<h_1\leq h_2, \; 0<\t_1\leq\t_2, \; 1<\s_1\leq \s_2,
$$
where $\hookrightarrow$ denotes the strict and dense inclusion,
and from Lemma \ref{osobineM_p_s} it follows that the norms given by
\eqref{Norma} and
\begin{equation} \label{EkvivalentneNorme}
\|\phi\|^{\sim}_{{\E}_{\tau, {\s},h}(K)}=\sup_{\alpha \in
\N^d}\sup_{x\in K}\frac{|\partial^{\alpha} \phi
(x)|}{h^{|\alpha|^{\s}}\lf|\alpha|^{\s}\rf!^{\frac{\tau}{\s}}}<\infty,\quad
\end{equation}
are equivalent in ${\E}_{\tau, {\s},h}(K)$. Moreover, instead of
$ \sup_{x\in K}|\partial^{\alpha} \phi (x)|$ we may put
$ \| \partial^{\alpha} \phi (x) \|_{L^p (K)},$ $ 1\leq p < \infty$
in \eqref{Norma} and \eqref{EkvivalentneNorme}.

By  ${\D}^K_{\t, \s,h}$ we denote the set of functions from ${\E}_{\t,
\s,h}(K)$ with support contained in $K$.
If  $U$ is an open set $\Rd$ and $ K \subset \subset U$ then we define families of spaces
by introducing the following projective and inductive limit topologies,
\begin{equation*}
\label{NewClassesInd} {\E}_{\{\t,
\s\}}(U)=\varprojlim_{K\subset\subset U}\varinjlim_{h\to
\infty}{\E}_{\t, {\s},h}(K),
\end{equation*}
\begin{equation*}
\label{NewClassesProj} {\E}_{(\t,
\s)}(U)=\varprojlim_{K\subset\subset U}\varprojlim_{h\to 0}{\E}_{\t,
{\s},h}(K),
\end{equation*}
\begin{equation*}
\label{NewClassesInd2} {\D}_{\{\t,
\s\}}(U)=\varinjlim_{K\subset\subset U} {\D}^K_{\{\t, \s\}}
=\varinjlim_{K\subset\subset U} (\varinjlim_{h\to\infty}{\D}^K_{\t,
\s,h})\,,
\end{equation*}
\begin{equation*}
\label{NewClassesProj2} {\D}_{(\t,
\s)}(U)=\varinjlim_{K\subset\subset U} {\D}^K_{(\t, \s)}
=\varinjlim_{K\subset\subset U} (\varprojlim_{h\to 0}{\D}^K_{\t,
\s,h}).
\end{equation*}
We will use abbreviated notation $ \t,\s $ for
$\{\t,\s\}$ or $(\t,\s)$.
The spaces ${\E}_{\t, \s}(U)$, ${\D}^K_{\t, \s}$ and ${\D}_{\t, \s}(U)$
are nuclear, cf. \cite{PTT-01}.

\par

If $ \t > 1 $ and $\s = 1$,
then $ {\E}_{\{\t, 1\}}(U)={\E}_{\{\t\}}(U)$  is the Gevrey class,
and $\D_{\{\t,1\}}(U)=\D_{\{\t\}}(U)$
is its subspace of compactly supported functions in $\E_{\{\t\}}(U)$. If $0<\t\leq 1$ then
$ {\E}_{\t, 1}(U)$ consists of quasianalytic functions. In particular, $\dss
\D_{\t,1}(U)=\{0\}$ when $0<\t\leq 1$, and ${\E}_{\{1, 1\}}(U)= {\E}_{\{1\}}(U)$ is the space of analytic functions on $U$.

The space $ \E_{\{1,2\}}(U)$ appears in the study of  strictly hyperbolic equations
where it describes the regularity of the coefficients in the space variable (with low regularity in time),
which is sufficient  to ensure that the corresponding Cauchy problem is well posed in appropriate solution spaces,
we refer to \cite{CL} for details.

\par

In the following Proposition we
capture the main embedding properties between the above introduced family of spaces.

\begin{prop}
\label{detectposition} \cite{PTT-02} Let $\s_1\geq 1$. Then for every $\s_2>\s_1$
and $\t>0$
\begin{equation*} \label{Theta_S_embedd}
\varinjlim_{\t\to \infty}{\E}_{\t,
{\s_1}}(U)\hookrightarrow \varprojlim_{\t\to 0^+} {\E}_{\t,
{\s_2}}(U).
\end{equation*}
Moreover, if $0<\t_1<\t_2$, then
\be \label{RoumieuBeurling} \E_{\{\t_1,\s\}}(U)\hookrightarrow
\E_{(\t_2,\s)}(U)\hookrightarrow \E_{\{\t_2,\s\}}(U), \;\;\; \s\geq 1,\nonumber \ee
and
$$
\varinjlim_{\t\to \infty}{\E}_{\{\t, {\s}\}}(U)= \varinjlim_{\t\to \infty} {\E}_{(\t, {\s})}(U),
$$
$$
\varprojlim_{\t\to 0^+}{\E}_{\{\t, {\s}\}}(U)= \varprojlim_{\t\to 0^+} {\E}_{(\t, {\s})}(U), \;\;\; \s\geq 1.
$$
\end{prop}

We conclude that
\begin{equation*}
\label{Theta_S_embedd-2}
{\E}_{\tau_0, {\s_1}}(U)\hookrightarrow \bigcap_{\tau> \tau_0} {\E}_{\tau, {\s_1}}(U)
\hookrightarrow {\E}_{\tau_0, {\s_2}}(U),
\end{equation*}
for any $\tau_0>0$ whenever $\s_2>\s_1\geq 1$, and
in particular,
\begin{equation*}
\label{GevreyNewclass}
\varinjlim_{t\to\infty} \E_{\{t\}}(U)\hookrightarrow {\E}_{\tau, \s}(U)
\hookrightarrow  C^{\infty}(U), \;\;\;
\tau>0, \; \s>1,
\end{equation*}
so that  the regularity in ${\E}_{\tau, \s}(U)$ can be thought of as an extended Gevrey regularity.

\par

Non-quasianalyticity condition $(M.3)'$ provides the existence of
partitions of unity in  $\E_{\{\t,\s\}}(U)$ which we formulate in the next Lemma.

\begin{lema}
\label{teoremaKompaktannosac}
Let $\t>0$ and $\s>1$. Then there exists a compactly supported function $\phi\in {\E_{\{\t,\s\}}}(U)$ such that $0\leq\phi\leq 1$ and $\int_{\Rd}\phi\,dx=1$.
\end{lema}

Compactly supported Gevrey function from $\E_{\{\t\}}(U)$ belong to $ {\D}_{\{\t, \s\}}(U)$.
However, in the proof of Lemma \ref{teoremaKompaktannosac} given in \cite{PTT-01} we
constructed a compactly supported function in ${\D}_{\{\t, \s\}}(U)$  which
does not belong to   ${\D}_{\{t\}}(U)$, for any $t>1$.

\begin{rem}
\label{notCarleman}
Note that the exponent $\s$ which appears in the power of term $h$ in \eqref{Norma}
makes the above definition different from the definition of Carleman class $C^L$, cf. \cite{HermanderKnjiga}.
This difference is essential for many calculations. For example, Carleman classes perform ``stability under differential operators``
since their defining sequences satisfy Komatsu's condition (M.2)'.
However, if $\tau>0$ and $\s>1$ then the sequence $M_p^{\tau,\s}$ does not satisfy  (M.2)'.
\end{rem}

If $\dss P=\sum_{|\alpha|\leq m}a_{\alpha}(x)\partial^{\alpha}$ is a partial differential operator of order $m$ with $a_{\alpha}\in \E_{\t,\s}(U)$, then  $P\,:\,\E_{\t,\s}(U)\to \E_{\t,\s}(U)$ is a continuous linear map with respect to the topology of $\E_{\t,\s}(U)$.
In particular, ${\E}_{\t, \s}(U)$ is closed under pointwise
multiplications and finite order differentiation, see \cite[Theorem 2.1]{TT}.

\par

Let $ \t > 0$, $ \s > 1$, and let $a_{\alpha} \in {\E}_{(\t, \s)}(U)$ (resp. $ a_{\alpha} \in {\E}_{\{\t, \s\}}(U)$) where $U$ is an open set in
$\Rd$. Then
$$
P(x,\partial)=\suml_{|\alpha|=0}^{\infty}a_{\alpha}(x){\partial}^{\alpha}
$$
is of class $(\t,\s)$ (resp. $\{\t,\s\}$) on  $U$ if for every $K\subset\subset U$ there exists constant $L>0$ such that for any $h>0$ there exists $A>0$ (resp. for every $K\subset\subset U$ there exists $h>0$ such that for any $L>0$ there exists $A>0$) such that,
\begin{equation*}
\label{Operatortausigma}
\sup_{x\in K}|\partial^{\beta}a_{\alpha}(x)|\leq A h^{{|\beta|}^{\s}}|\beta|^{\t{|\beta|}^{\s}}\frac{L^{|\alpha|^{\s}}}{|\alpha|^{\t 2^{\s-1}{|\alpha|}^{\s}}},\quad {\alpha,\beta \in \N^d}.
\end{equation*}

If $ \t >1$  and $ \s =1,$ then $ P(x,\partial)$ of class $(\t,1)$ (resp. $\{\t,1\}$) is
Komatsu's ultradifferentiable operator of class $(p!^{\t})$ (resp. $\{p!^{\t}\}$), see \cite{KomatsuNotes}.

The following theorem gives the continuity properties
of such differential operators on  ${\E}_{\tau, \s}(U)$, cf. \cite[Theorem 2.1]{PTT-02} for the proof.

\begin{te}
\label{TeoremaZatvorenostUltraDfOP}
Let $P(x,\partial)$ be a differential operator of class $(\t,\s)$ (resp. $\{\t,\s\}$). Then
\be
P(x,\partial):\quad {\E}_{\t, \s}(U) \longrightarrow {\E}_{\t 2^{\s-1}, \s}(U)\nonumber
\ee
is a continuous linear mapping, and the same holds for
$$
P(x,\partial):\quad \varinjlim_{\t\to \infty}\E_{\t,\s}(U) \longrightarrow \varinjlim_{\t\to \infty}\E_{\t,\s}(U).
$$
\end{te}

\subsection{Superposition in $\E_{\t,\s} (U)$}\label{subsecInverse}

In this subsection we prove that the classes $\E_{\t,\s} (U)$,  $\t>0$, $\s>1$, are
stable under \emph{superposition}, and conclude that
they are inverse closed. We  refer to \cite{Mostefa, RS, CarmenAntonio} for related results.
We emphasize here that the inverse-closedness of $\E_{\t,\s}(U)$
plays an essential role in the proof our main result, Theorem \ref{FundamentalTheoremNK}.

\par

Recall, an algebra $\mathcal A$ is {\it inverse-closed} in $C^{\infty}(U)$ if for any $\varphi\in {\mathcal A}$ for which $\varphi(x)\not=0$ on $U$ it follows that $\dss \varphi ^{-1} \in {\mathcal A}$.
It is proved in \cite{Siddiqi} that a Carleman class defined by a sequence $M_p$ is inverse closed in  $C^{\infty}(U)$
if  there exists $C>0$ such that
\be
\label{UsloviInveseClosed}
\left ( \frac{M_p}{p!}\right )^{1/p} \leq C \left ( \frac{M_q}{q!}\right )^{1/q}, \;\;\;  p \leq q,\quad \text{and}\quad \lim_{p\to\infty} M_p^{1/p}=\infty,
\ee
where the condition on the left hand side of \eqref{UsloviInveseClosed} is equivalent to the statement that $ ( M_p /p! )^{1/p} $ is an almost increasing sequence.

The Stirling formula implies that the sequence
$ ( M_p /p! )^{1/p} $ is almost increasing  if and only if
$$
\frac{M_p ^{1/p}}{p} \leq C  \frac{M_q ^{1/q}}{q}, \;\;\;  p \leq q.
$$

For example, $\E_{\{\t\}}(U)$, $ \t \geq 1$ are inverse-closed algebras.

Since  $\dss \Big(\frac{M_p ^{\t,\s}}{p^p}\Big)^{1/p}=p^{\t p^{\s-1}-1}$ when $M_p^{\t,\s}=p^{\t p^{\s}}$, $\t>0$, $\s>1$,
and
$$p^{\t p^{\s-1}-1}<q^{\t q^{\s-1}-1},\quad \lceil (1/\t)^{1/(\s-1)}\rceil < p <q,$$
we conclude that $\dss \Big(\frac{M_p ^{\t,\s}}{p^p}\Big)^{1/p}$ is an almost increasing sequence and
for any choice of indices $ k_i,$ $ i =1, \dots, j,$ and $k=\sum_{i=1}^j k_i,$ we have
\be
\label{NejednakostAIseq}
\frac{M_{k_i}^{\t,\s}}{k_i!}\leq C^{k_i} \left ( \frac{M_{k}^{\t,\s}}{k!} \right)^{k_i/k}, \;\;\;
\text{so that} \;\;\;
\prod_{i=1}^j\frac{M_{k_i}^{\t,\s}}{k_i!}\leq C^k \frac{M_{k}^{\t,\s}}{k!}. \ee
In other words
$$
{\prod_{i=1}^j }k_i^{\t k_i^{\s}}\leq C^{k}\frac{k_1!\cdot\cdot\cdot k_j!}{k!}k^{\t k^{\s}}, \;\;\; k=\sum_{i=1}^j k_i.
$$

The almost increasing property of defining sequences is used in the proofs of inverse closedness in Carleman classes, see
\cite{rudin, Siddiqi, Klotz}.

Instead, we prove more general result on superposition. We will use  Fa\'a di Bruno formula as presented in \cite{Tsoy}.
Let us first fix the notation.
A multiindex $\alpha\in \N^d$ is said to be \emph{decomposed into parts} $p_1,\dots,p_s\in \N^d$ with \emph{multiplicities} $m_1,\dots,m_s\in \N$, respectively, if
\be
\label{oblikParticijaFaa}
\alpha=m_1 p_1+m_2 p_2+\dots+m_s p_s,
\ee where $m_i\in\{0,1,\dots, |\alpha|\}$, $|p_i|\in\{1,\dots,|\alpha|\}$, $i =1,\dots,s$.

\par

If $p_i=(p_{i_1},\dots, p_{i_d})$, $i\in \{1,\dots,s\}$, we put $p_i<p_j$ when $i<j$, that is when there exists $k\in \{1,\dots,d\}$ such that $p_{i_1}=p_{j_1},\dots, p_{i_{k-1}}=p_{j_{k-1}}$ and $p_{i_{k}}<p_{j_{k}}$.

Note that $s\leq |\alpha|$ and the same holds for the \emph{total multiplicity} $m= m_1+\dots+m_s \leq |\alpha|$.

Any decomposition of $ \alpha $ can be therefore identified with the triple  $(s,p,m)$,
and the set of all decompositions of the form \eqref{oblikParticijaFaa} is denoted by $\pi$.
The total number $\card \pi$ of decompositions given by
\eqref{oblikParticijaFaa}  is bounded by $(1+|\alpha|)^{d+2}.$

For smooth functions $f: U\to {\mathbf C}$ and $g: V\to U$, where $U,V$ are open in $\mathbf R$ and  $\Rd$, respectively, the generalized Faa di Bruno formula is given by
\be
\label{GeneralizedFaaDiBruno}
\partial^{\alpha}(f(g))=\alpha! \sum_{(s,p,m)\in \pi} f^{(m)}(g)\prod_{k=1}^{s}\frac{1}{m_k!}\Big(\frac{1}{p_k !}\partial^{p_k} g \Big)^{m_k}.
\ee

\par

We say that the sequence $M_p$, $p\in \N$ of positive numbers satisfies \emph{Fa\'a di Bruno property} if there exist a constant $C>0$ such that for every $j\in \Z$ and $k_i\in \Z$  we have
\be
\label{NejednakostAIseq1}
M_j\prod_{i=1}^j M_{k_i} \leq C^{ \sum_{i=1}^{j} k_i} M_{ \sum_{i=1}^{j} k_i}.
\ee

By \cite[Lemma 2.2]{RS} it follows that if $M_p$, $p\in \N$ satisfies $(M.2)'$ and if  $M_p^{1/p}$ is almost increasing, then $M_p$ satisfies
Fa\'a di Bruno property. Since $M_p^{\t,\s}=p^{\t p^{\s}}$, $\t>0$, $\s>1$, does not satisfy $(M.2)'$ we
first prove a modified version
of Fa\'a di Bruno property for the sequence $\dss \frac{M_p^{\t,\s}}{p!}$, $p\in \N$.

\begin{lema}
\label{modifiedFaa}
Let there be given $\t>0$, $\s>1$ and let $M_p^{\t,\s}=p^{\t p^{\s}}$, $p\in\N$. Then there exist a constant $C>0$
such that for every $j\in \Z$ and
$k_i\in \Z$, $i=1,\dots,j,$ we have
\be
\label{NejednakostAIseq2}
\frac{M_j^{\t,\s}}{j!}\prod_{i=1}^j \frac{M_{k_i}^{\t,\s}}{k_i !} \leq C^{k^{\s}} \frac{M_k^{\t,\s}}{k!},
\ee
where $\sum_{i=1} ^j k_i = k.$
\end{lema}

\begin{proof}
We follow the ideas from  the proof of \cite[Theorem 4.11.]{RS}.

First we note that the assertion is trivial if $j=k$ since then $k_i=1$ for all $1\leq i \leq j$ and therefore
$$
\frac{M_j^{\t,\s}}{j!} \prod_{i=1}^j \frac{M_{k_i}^{\t,\s}}{k_i !} =
\frac{M_k^{\t,\s}}{k!} \Big(\frac{M_1^{\t,\s}}{1!}\Big)^k=\frac{M_k^{\t,\s}}{k!}.
$$

For $j<k$, set $I=\{i\,|\,\,1\leq i\leq j,\,k_i\geq 2\}$ and $\tilde{k_{i}}=k_i -1$, $i\in I$. Note that
\be
\label{NejednakostKomp}
k=\sum_{i=1}^{j}k_i=\sum_{i\in I}k_i+\sum_{i\not\in I,\,1\leq i \leq j}k_i
=\sum_{i\in I}k_i+j-\card I=\sum_{i\in I}\tilde{k_i}+j,
\ee
and since $\dss \Big(\frac{M_p^{\t,\s}}{p!}\Big)^{1/p}$ is almost increasing, then the inequality \eqref{NejednakostAIseq} implies
that
\be
\label{NejednakostKomp1}
\frac{M_j^{\t,\s}}{j!}\prod_{i\in I} \frac{M_{\tilde{k_i}}^{\t,\s}}{\tilde{k_i} !}\leq C^{k} \frac{M_k^{\t,\s}}{k!},
\ee

Moreover, from $\overline{(M.2)'}$  and  $k_i=\tilde{k_i}+1$, $i\in I$, we obtain
\be
\label{NejednakostKomp2}
\frac{M_{k_i}^{\t,\s}}{k_i !} \leq C_1^{\tilde{k_i}^{\s}} \frac{M_{\tilde{k_i}}^{\t,\s}}{\tilde{k_i} !},
\ee
for some constant $C_1>0.$

By combining \eqref{NejednakostKomp}, \eqref{NejednakostKomp1} and \eqref{NejednakostKomp2} we obtain
$$
\frac{M_j^{\t,\s}}{j!} \prod_{i=1}^j \frac{M_{k_i}^{\t,\s}}{k_i !}\leq  \Big(\frac{M_1^{\t,\s}}{1!}\Big)^{j-\card I} \frac{M_j^{\t,\s}}{j!}
\prod_{i\in I} \frac{M_{{k_i}}^{\t,\s}}{{k_i}!}\leq \frac{M_j^{\t,\s}}{j!}\prod_{i\in I}C_1^{\tilde{k_i}^{\s}}
\frac{M_{\tilde{k_i}}^{\t,\s}}{\tilde{k_i} !}
$$
$$
\leq C_1^{(k-j)^{\s}}
\frac{M_j^{\t,\s}}{j!}\prod_{i\in I} \frac{M_{\tilde{k_i}}^{\t,\s}}{\tilde{k_i} !}\leq C_2^{ k^{\s}}\frac{M_k^{\t,\s}}{k!},
$$
for some constant $C_2>0$ and the Lemma is proved.

\end{proof}

The main result of this section reads as follows.

\begin{te}
\label{TeoremaKompozicija}
Let there be given $\t>0$, $\s>1$, and let $U$ and $V$ be open sets in $\mathbf R$ and $\Rd$, respectively.
If  $f\in \E_{\t,\s}(U)$ and
$g\in \E_{\t,\s}(V)$ is such that $g: V\to U$, then $f\circ g\in \E_{\t,\s}(V)$.
\end{te}

\begin{proof} For simplicity we show that
if  $f\in \E_{\{\t,\s\}}(U)$ and
$g\in \E_{\{\t,\s \}}(V)$ is such that $g: V\to U$, then $f\circ g\in \E_{\{\t,\s\}}(V)$,
and leave the (so-called Beurling) case $f\in \E_{(\t,\s)}(U)$ and
$g\in \E_{(\t,\s )}(V)$ to the reader.

Let $K\subset\subset V$ and $h>0$ be fixed so that $g\in  \E_{\t,\s,h}(K)$.
Put $I=\{g(x), x\in K\}$ and note that $I$ is a compact set, $I \subset \subset U$. Therefore
$f\in  \E_{\t,\s,h'}(I)$ for some  $h'>0$.
By the Fa\'a di Bruno formula \eqref{GeneralizedFaaDiBruno}, for any $x\in K$ we have the following estimate
\begin{multline}
\label{FaaDibrunoKomp}
|\partial^{\alpha}(f\circ g)(x)|\leq|\alpha|! \sum_{(s,p,m)\in \pi} |f^{(m)}(g(x))|\prod_{k=1}^{s}\frac{1}{m_k!}\Big(\frac{1}{p_k !}|\partial^{p_k} g(x)| \Big)^{m_k}\\
\leq A^{|\alpha|+1} |\alpha|! \sum_{(s,p,m)\in \pi} \Big(h^{' m^{\s}}\prod_{k=1}^{s}h^{m_k|p_k|^{\s}}\Big)\frac{m!}{m_1!\dots m_s!}\frac{m^{\t m^{\s}}}{m!}\prod_{k=1}^{s}\Big(\frac{|p_k|^{\t |p_k|^{\s}}}{|p_k| !}\Big)^{m_k}
\end{multline}
for some $A>0$,
and the second sum being taken over all decompositions $\dss |\alpha|=\sum_{k=1}^s m_k |p_k|$ where $\dss m=\sum_{k=1}^s m_k$, $m_k\in\{0,1,\dots, |\alpha|\}$, $|p_k|\in\{1,\dots,|\alpha|\}$, $k =1,\dots,s$ and $s\leq |\alpha|$.

By Lemma \ref{modifiedFaa} we have
\be
\label{NejednakostAIseq3}
\frac{ m^{\t m^{\s}}}{m!}\prod_{k=1}^{s}\Big(\frac{|p_k|^{\t |p_k|^{\s}}}{|p_k| !}\Big)^{m_k} \leq C^{|\alpha|^{\s}} \frac{|\alpha|^{\t |\alpha|^{\s}}}{|\alpha|!}.
\ee
Moreover,
$$
m^{\s}+\sum_{k=1}^{s}m_k|p_k|^{\s}\leq |\alpha|^{\s}+ |\alpha|^{\s-1}\sum_{k=1}^{s}m_k |p_k|=2 |\alpha|^{\s}
$$
wherefrom
\be
\label{NejednakostAIseq4}
h^{' m^{\s}}\prod_{k=1}^{s}h^{m_k|p_k|^{\s}}\leq C_1^{ m^{\s}+\sum_{k=1}^{s}m_k|p_k|^{\s}}\leq C_1^{2 |\alpha|^{\s}},
\ee
where $C_1=\max\{h, h'\}$.
From \eqref{NejednakostAIseq3}, \eqref{NejednakostAIseq4} and \eqref{FaaDibrunoKomp} we conclude
that there is a constant  $C_2>0$ such that
\be
\label{NejednakostAIseq5}
|\partial^{\alpha}(f\circ g)(x)|\leq C_2^{ |\alpha|^{\s}+1}|\alpha|^{\t |\alpha|^{\s}} \sum_{(s,p,m)\in \pi} \frac{m!}{m_1!\dots m_s!}, \quad x\in K.
\ee

It remains to estimate $\dss\sum \frac{m!}{m_1!\dots m_s!}$. Note that without loss of generality we may assume that $s=|\alpha|$ (for $s<|\alpha|$ we may put $m_k=0$, for $s<k\leq |\alpha|$). Since $|p_k|\in\{1,\dots,|\alpha|\}$ note that we can write

$$|\alpha|=\sum_{k=1}^{|\alpha|} m_k |p_k|=\sum_{k=1}^{|\alpha|}k m'_k,$$ where $m=\sum_{k=1}^{|\alpha|}m'_k$. Hence we conclude that the summation in \eqref{NejednakostAIseq5} can be taken over all  $(m_1,\dots,m_s)\in \N^s$, $s=|\alpha|$, such that $\dss |\alpha|=\sum_{k=1}^{|\alpha|}k m_k$ and $\dss m=\sum_{k=1}^{|\alpha|}m_k$. Therefore,
$$\sum \frac{m!}{m_1!\dots m_s!}=2^{1m_1+2m_2+\dots+|\alpha|m_{|\alpha|} -1}=2^{|\alpha|-1},$$ and the proof is completed.
\end{proof}

As an immediate consequence of Theorem \ref{TeoremaKompozicija} we conclude the following:

\begin{cor}
\label{teoremaInverseClosedness}
Let $U\subseteq \Rd$ be open. Classes $\E_{\t,\s}(U)$, $\t>0$, $\s>1$, are inverse-closed in $C^{\infty}(U)$.
\end{cor}

Note that the proof of Theorem \eqref{TeoremaKompozicija} holds even if $\s=1$ and $\t\geq 1$, so that
we recover the well known results on stability under superposition of Gevrey (analytic) classes of functions (see \cite{Mostefa, RS, CarmenAntonio, Klotz}).

\section{Wave front sets related to classes $\E_{\t,\s}$}\label{ChapterWaveFront}

Let $\t>0$, $\s>1$, $\Omega\subseteq K\subset \subset U\subseteq\Rd$, where $\Omega$ and $U$ are open in $\Rd$, $K$ is compact in $\Rd$,
in and the closure of $\Omega$ is contained in $K$, $ \overline{\Omega} \subseteq K$.

Let $u\in \D'(U)$. In \cite{PTT-01} we investigated the nature of regularity related to the condition
 \begin{equation}
\label{uslov3'}
|\widehat u_N(\xi)|\leq A\, \frac{h^{N} N!^{\t/\s}}{|\xi|^{\lfloor N^{1/\s} \rfloor}}, \quad N\in { \N},\,\xi\in \Rd\backslash\{0\}.
\end{equation} where $\{u_N\}_{N\in \N}$ is bounded sequence in $\E'(U)$ such that $u_N=u$ in $\Omega$ and $A,h$ are some positive constants.

Note that the conditions  (\ref{uslov3'}) can be replaced by an equivalent set of conditions if instead of  $N$ we use another
positive, increasing sequence $a_N$ such that $a_N\to \infty$, $N\to \infty$ (cf. \cite{PTT-02}).
This change of variables called {\em enumeration}, ``speeds up`` or ``slows down`` the decay estimates of
single members of the corresponding sequences,
without changing the asymptotic behavior of the whole sequence when $N \rightarrow \infty$.
After applying the enumeration $N\to a_N$ we can write  again $u_N$ instead of $u_{a_N}$, since we are only interested in the  asymptotic behavior.

For example, Stirling's formula  and enumeration $N\to N^{\s}$ applied to (\ref{uslov3'}) give an equivalent  estimate of the form
\be
\label{uslov3'''}
|\widehat u_N(\xi)|\leq A_1\, \frac{h_1^{N^{\s}} N^{\t N^{\s}}}{|\xi|^{N}},
\quad N\in { \N},\,\xi\in \Rd\backslash\{0\},
\ee
for some constants $A_1, h_1 > 0$.

Wave-front sets  ${\WF}_{\{\t,\s\}}(u)$ (see Remark \ref{BjorlingRem} for ${\WF}_{(\t,\s)}(u)$) are introduced in \cite{PTT-02}
in the study of local regularity in ${\E}_{\{\tau, \s\}}(U)$ .
Together with enumeration we used sequences of cutoff functions in a similar way as it is done in
\cite{HermanderKnjiga} in the context of  analytic wave front set ${\WF}_A$.
We recall the definition of ${\WF}_{\{\t,\s\}}(u)$.

\begin{de}
\label{Wf_t_s1}
Let  there be given $u\in \D'(U)$, $\t>0$, $\s>1$, and $(x_0,\xi_0)\in U\times\Rd\backslash\{0\}$. Then $(x_0,\xi_0)\not \in {\WF}_{\{\t,\s\}}(u)$ if there exists an open neighborhood $\Omega $ of $x_0$, a conic neighborhood $\Gamma$ of $\xi_0$ and a bounded sequence $\{u_N\}_{N\in \N}$ in $\E'(U)$ such that $u_N=u$ on $\Omega$ and \eqref{uslov3'} holds for all $\xi \in \Gamma$ and for some constants $A,h>0$.
\end{de}

For a given $u\in \D'(U)$ it immediately follows  that ${\WF}_{\{\t,\s\}}(u)$
is closed subset of $U\times\Rd\backslash\{0\}$. Note that for $\t>0$ and $\s>1$
$$
{\WF}_{\{\t,\s\}}(u)\subseteq{\WF}_{\{1,1\}}(u)={\WF}_A (u),\;\;\;
u\in \D'(U),
$$
where ${\WF}_A (u)$ denoted the analytic wave front set of a distribution $u\in \D'(U)$, cf. \cite{HermanderKnjiga}.

Next we prove that in the definition of
${\WF}_{\{\t,\s\}}(u)$ a  bounded sequence  of cut-off functions $\{u_N\}_{N\in \N} \subset \E'(U)$  can be replaced by a single function from $\D_{\{\t,\s\}}(U)$. First we give an example of $\phi\in \D_{\{\t,\s\}}(U)$ such that
$\phi = 1$ on particular open sets.

\begin{ex}
\label{FiCutOff} Let there be given $x_0\in \Rd$, $\t >0,$ $\s >1$, and let $\dss d=\sum_{p=1}^{\infty}\frac{1}{(2(p+1))^{\t {p^{{\s}-1}}}}$.
By Lemma \ref{teoremaKompaktannosac} and  \cite[Theorem 1.4.2]{HermanderKnjiga}, there exists  $\psi\in \D^{\overline{B_{d/2}(x_0)}}_{\{\t,\s\}}$ such that $\dss \int\psi (x)\, dx=1$. If $\chi$ denotes the characteristic function of
$$\dss \{y\in \Rd\,|\, |x-y|\leq d/2,\, x\in\overline{B_{d/2}(x_0)}\},$$ then
$\phi=\chi*\psi =1$ on an open neighborhood $\Omega$ of $\overline{B_{d/2}(x_0)}$. In particular, if $U$ is  an open set such that
$$\inf\{|x-y|\,:\,x\in U^{c},\, y\in \overline{B_{d/2}(x_0)} \}>d$$
then $\phi\in \D_{\{\t,\s\}}(U)$.
\end{ex}

\begin{rem}
\label{PaleyRemark}
In the sequel we will use the following  Paley-Wiener type  estimates.
If  $u\in\E'(U)$, then
$|\widehat u(\xi)|\leq C\langle\xi\rangle^M,\,\xi\in \Rd,$
for some constant $C>0$, where $M$ denotes the order of distribution $u$.

Similarly, if $\phi\in \D^K_{\{\t,\s\}}$, where $K$ is a compact set in $ \Rd$,
then
\be
\label{chiNPaley}
|\widehat \phi (\xi)|\leq A h^{|\alpha|^\s} {|\alpha|^{\t |\alpha|^{\s}}}\langle\xi\rangle^{-|\alpha|},\quad
\alpha\in {\N^d}, \xi\in \Rd,
\ee
for some constants $A,h>0$.
\end{rem}
\begin{te}
\label{TeoremaKarakterizacija}
Let  $u\in \D'(U)$,  $\t>0$, $\s>1$, and let $(x_0,\xi_0)\in U\times\Rd\backslash\{0\}$.
Then $(x_0,\xi_0)\not \in \WF_{\{\t,\s\}}(u)$ if and only if there exists a conic
neighborhood $\Gamma_0$ of $\xi_0$, a compact set
$ K\subset\subset U $ and
$\phi\in \D_{\{\t,\s\}}^K$ such that $\phi=1$ on a neighborhood of $x_0$, and such that
\be
\label{SingsuplemaUslov1}
|\widehat{\phi u}(\xi)|\leq A \frac{h^{N^{\s}} N^{\t N^{\s}}}{|\xi|^N},\quad N\in {\N}\,,\xi\in \Gamma_0\,,
\ee
for some $A,h>0$.
\end{te}

\begin{proof}
The necessity is trivial, since if there is $\phi\in \D_{\{\t,\s\}}^K$, $ K\subset\subset U $, $\phi=1$ on a neighborhood $\Omega$ of $x_0$
and such that \eqref{SingsuplemaUslov1} holds
in a conic neighborhood $\Gamma_0$ of $\xi_0$, then by putting
$u_N=\phi u$, for every $N\in \N$ it follows that $(x_0,\xi_0)\not \in \WF_{\{\t,\s\}}(u)$.

\par

Now assume that $(x_0,\xi_0)\not \in \WF_{\{\t,\s\}}(u)$, i.e. that there exists an open neighborhood $\Omega$ of $x_0$, a conic neighborhood $\Gamma$ of $\xi_0$ and a bounded
sequence $\{u_N\}_{N\in \N}$ in $\E'(U)$ such that $u_N=u$ on $\Omega$ and such that
\begin{equation}
\label{uslovEnumeracija}
|\widehat{u_N} (\xi)|\leq A \frac{h^{N^{\s}} N^{\t N^{\s}}}{|\xi|^N},
\quad N\in { \N},\,\xi\in \Gamma.
\end{equation}

Choose $\phi\in \D_{\{\t,\s\}}^{K_{x_0}}$, $K_{x_0}\subset\subset \Omega$, $\phi=1$ on some neighborhood of $x_0$, and choose a conic neighborhood $\Gamma_0$ of $\xi_0$ with the closure contained in $\Gamma$.
Let $\varepsilon >0$ be chosen so that $\xi-\eta \in \Gamma$ when $\xi\in \Gamma_0$ and $|\eta|<\varepsilon |\xi|$.

Since $\phi u=\phi u_N$,
$$
\widehat {\phi u}(\xi)=\Big(\int_{|\eta|<\varepsilon |\xi|} +\int_{|\eta|\geq\varepsilon |\xi|}\Big)
\widehat \phi (\eta)\widehat u_N(\xi-\eta)\,d\eta =I_1+I_2\,,\quad \xi\in \Gamma_0.
$$

To estimate $I_1$ we use that $|\eta|<\varepsilon |\xi|$ implies $|\xi-\eta|\geq |\xi|-|\eta|> (1-\varepsilon)|\xi|.$
By (\ref{uslovEnumeracija}) and $|\widehat \phi (\eta)|\leq B\langle\eta\rangle^{-d-1}$ for some $B>0$,
we have
\begin{multline}
|I_1|=\Big| \int_{|\eta|<\varepsilon |\xi|} \widehat \phi(\eta)\widehat{u_N}(\xi-\eta)\,d\eta \Big| \\[1ex]
\leq \int_{|\eta|<\varepsilon |\xi|} |{\widehat \phi} (\eta)| A \frac{h^{N^{\s}} N^{\t N^{\s}}}{|\xi-\eta|^{N}} d\eta
\\[1ex]
\leq A B \frac{h^{N^{\s}} N^{\t N^{\s}}}{{((1-\varepsilon)|\xi|)}^{N}} \int_{{\bf R}^d}\langle \eta \rangle^{-d-1} d\eta
\\[1ex]
\leq A_1 \frac{h_1^{N^{\s}} N^{\t N^{\s}}}{|\xi|^{N}} ,\quad \xi\in \Gamma_0, N\in \N,
\end{multline}
for some constants $A_1, h_1>0$. For the last estimate we have used $\dss (1-\varepsilon)^{-N}<(1-\varepsilon)^{-N^{\s}}$ when $\s>1$.

To estimate $I_2$ we use that  $|\eta|\geq \varepsilon |\xi|$ implies $ |\xi-\eta|\leq |\xi|+|\eta|\leq (1+1/\varepsilon)|\eta|. $
For a given $N\in \N$, we put $|\alpha|={N+M+d+1}$,
where $M>0$ is the order of distribution $u$. Then by \eqref{chiNPaley} there exist constants
$A,h>0$ such that
\begin{multline}
|I_2|=\Big| \int_{|\eta|\geq \varepsilon |\xi|} \widehat \phi(\eta)\widehat{u_N}(\xi-\eta)\,d\eta \Big|\\
\leq\frac{A h^{{(N+M+d+1)}^\s} {(N+M+d+1)^{\t {(N+M+d+1)}^{\s}}}}{(\varepsilon |\xi|)^{N}}\\
\int_{|\eta|\geq \varepsilon |\xi|} \langle\eta\rangle^{-M-d-1} C \langle\xi-\eta\rangle^{M} \,d\eta\\
\leq \frac{A_1 h_1^{N^\s} {N^{\t N^{\s}}}}{|\xi|^{N}} \quad \xi\in \Gamma_0, N\in \N,
\end{multline}
where $h_1=\max\{h, h^{2^{\s-1}}\},$ $ A_1 $ $ =A\, \max\{1,h^{2^{\s-1}(M+d+1)}\}.$

In the last inequality we used
$$
\dss |\alpha|^{\s}+|\beta|^{\s}\leq |\alpha+\beta|^{\s}\leq 2^{\s-1}(|\alpha|^{\s}+|\beta|^{\s}),\quad \alpha,\beta \in \N^d,
$$
and $\overline{(M.2)'}$ property of  $M_p^{\t,\s}=p^{\t p^{\s}}$.

Thus, \eqref{SingsuplemaUslov1} follows and the theorem is proved.
\end{proof}

\begin{rem}
\label{BjorlingRem}
In the Beurling case, for  $u\in \D'(U)$, $\t>0$, $\s>1$, and  $(x_0,\xi_0)\in U\times\Rd\backslash\{0\}$ we have that
$(x_0,\xi_0)\not \in {\WF}_{(\t,\s)}(u)$ if there exists open neighborhood $\Omega $ of $x_0$, a
conic neighborhood $\Gamma$ of $\xi_0$ and a bounded sequence $\{u_N\}_{N\in \N}$ in $\E'(U)$ such that $u_N=u$ on $\Omega$ and such that
for every $h>0$ there exists $A>0$ such that
\begin{equation*}
|\widehat u_N(\xi)|\leq A\, \frac{h^{N} N!^{\t/\s}}{|\xi|^{\lfloor N^{1/\s} \rfloor}}, \quad N\in { \N},\,\xi\in \Gamma.
\end{equation*}

Note that Theorem \ref{TeoremaKarakterizacija} can be formulated for the Beurling case as well with $\phi\in \D_{(\t,\s)}^K$ such that
\eqref{SingsuplemaUslov1} holds for  every $h>0$ and for some $A = A(h)>0$.
More precisely, for any $h>0$ we can choose $\phi\in \D_{\t,\s,C_h}^K$ where $C_h=\min\{h, h^{\frac{1}{2^{\s-1}}}\}$ and obtain
$\phi\in \D_{(\t,\s)}^K$ with the desired properties.

Thus the results concerning ${\WF}_{(\t,\s)}(u)$ are analogous to those for ${\WF}_{\{\t,\s\}}(u)$,
and we will consider only the later wave-front sets in the sequel.
\end{rem}

We end this section an auxiliary result which will be used in the proof of  Theorem \ref{FundamentalTheoremNK}.

\par

\begin{lema}
\label{LemaGlavnaTeorema}
Let $u\in\D'(U)$, $\t>0$, $\s>1$, $\Omega\subset K\subset\subset U$, where $U$ and $\Omega$ are open.
If $F$ is a closed cone such that ${\WF}_{\{\t,\s\}}(u) \cap (K\times F)=\emptyset$
and $\phi\in \D_{\{\t,\s\}}^K$, $\phi=1$ on $\Omega$, then for some $A,h>0$ it holds
\be
\label{LemaUslov}
|\widehat{\phi u}(\xi)|\leq A \frac{h^{N^{\s}} N^{\t N^{\s}}}{|\xi|^N},\quad N\in {\N}\,,\xi\in F\,.
\ee
\end{lema}

\begin{proof}
Let $(x_0,\xi_0)\in K\times F$, and set $r_0:=r_{{x_0,\xi_0}}>0$.
Furthermore, let $\phi\in \D_{\{\t,\s\}}(B_{r_0}(x_0))$, $\overline{B_{r_0}(x_0)}\subseteq \Omega\subseteq K$.

Since $(x_0,\xi_0)\not \in {\WF}_{\{\t,\s\}}(u)$ by Theorem \ref{TeoremaKarakterizacija}
there exists $\psi\in \D_{\{\t,\s\}}(U)$, $\psi=1$ on $\Omega$, and  a conical neighborhood $\Gamma$ of $\xi_0$, such that
\begin{equation}
|\widehat{\psi u}(\xi)|\leq A \frac{h^{N^{\s}} N^{\t N^{\s}}}{|\xi|^N},
\quad N\in { \N},\,\xi\in \Gamma,
\end{equation}
for some $A,h>0$.

Let $\Gamma_0$ be an open conical neighborhood of $\xi_0$ with the closure contained in $\Gamma$. We write
$$
\widehat{\phi u}(\xi)=\Big(\int_{|\eta|<\varepsilon |\xi|} +\int_{|\eta|\geq\varepsilon |\xi|}\Big)
\widehat \phi (\eta)\widehat{\psi u}(\xi-\eta)\,d\eta =I_1+I_2\,,\quad \xi\in \Gamma_0,\,,
$$ and arguing in a similar way as in the proof of Theorem \ref{TeoremaKarakterizacija} we obtain  \eqref{LemaUslov} for $(x,\xi)\in B_{r_0}(x_0)\times \Gamma_0$.

In order to extend the result to $K\times F$ we use the same idea
as in the proof of \cite[Lemma 8.4.4]{HermanderKnjiga}.
Since the intersection of $F$ with the unit sphere is a compact set,
there exists a finite number $n$ of balls $B_{r_{{x_0,\xi_j}}}(x_0)$,
such that $F \subset \cup_{j=1} ^n \Gamma_j$. Note that \eqref{LemaUslov} remains true if $\phi$ is chosen so that $\dss \supp\phi \subseteq B_{r_{x_0}}:=\bigcap_{j=1}^{n}B_{r_{x_0,\xi_j}}(x_0)$, $\xi_j\in \Gamma_j$.

Moreover, since $K$ is  compact set, it can be covered by a finite number $m$ of balls $B_{r_{x_k}}$, $k\leq m$.
Since $M_p^{\t,\s}=p^{\t p^{\s}}$ satisfies $(M.1)$ and $(M.3)'$, then  there exist
non-negative functions $\phi_k\in \D_{\{\t,\s\}}(B_{r_{x_k}})$, $k\leq n$, such that $\suml_{k=1}^n \phi_k=1$
on a neighborhood of $K$ (cf. \cite[Lemma 5.1.]{Komatsuultra1}).

To conclude the proof we note that if $\phi\in \D^K_{\{\t,\s\}}$ then $\phi\phi_k\in \D_{\{\t,\s\}}(B_{r_{x_k}})$ and consequently \eqref{LemaUslov} holds if we replace $\phi$ by $\phi \phi_k$. Since $\suml_{k=1}^n\phi\phi_k=\phi$, the proof is finished.
\end{proof}

\section{Main result}\label{SectionNonConstant}

We first recall the definition of the characteristic set of an operator and the main property of its principal symbol, cf.
\cite{Rauch}.

If $ P(x,D)=\sum_{|\alpha|\leq m} a_{\alpha } (x) D^{\alpha}$ is a differential operator of order $m$ on $U$ and
$a_\alpha \in C^\infty (U)$, $|\alpha|\leq m $,  then
its characteristic variety at $\overline{x}\in U$  is given by
$${\rm Char}_{\overline{x}}(P)=\{(\overline{x},\xi)\in U\times \Rd\backslash\{0\} \,|\, P_m(\overline{x},\xi)=0 \},$$
and its characteristic set  on $U$ is given by
$${\rm Char}(P)=\bigcup_{\overline{x}\in U}{{\rm Char}_{{\overline{x}}}(P)}.$$
Here $\dss P_m(x,\xi)=\sum_{|\alpha|=m}a_{\alpha}(x)\xi^{\alpha} \in C^{\infty}(U\times\Rd\backslash\{0\})$
is the principal symbol of $ P(x,D).$

By the homogeneity of the principal symbol it follows that ${\rm Char}(P)$ is a closed conical subset of $U\times \Rd\backslash\{0\}$.

If $(x_0,\xi_0)\not \in {\rm Char}(P)$ then there exists an open neighborhood
$\Omega$ of $x_0$ and a conical neighborhood $\Gamma$ of $\xi_0$ such that $P_{m}(x,\xi)\not=0$, $x\in \Omega$ and $\xi\in \Gamma$. Moreover,
since the principal symbol is homogeneous we have
\be
\Big|P_m(x,\frac{\xi}{|\xi|})\Big|=\frac{1}{|\xi|^{m}}|P_m(x,\xi)|\geq C,\quad x\in \Omega, \xi\in \Gamma,\nonumber
\ee
so that for any compact set $K\subset \subset \Omega$ there are constants  $0<C_1 < C_2$ such that
\be
\label{OcenaSaDonjeStrane}
C_1 |\xi|^m \leq |P_m(x,\xi)|\leq C_2|\xi|^{m},\quad x\in K, \xi\in \Gamma.\nonumber
\ee

The main result of this section, Theorem \ref{FundamentalTheoremNK} extends \cite[Theorem 1.1]{PTT-02}
to operators with variable coefficients. We recall that
in \cite[Theorem 8.6.1]{HermanderKnjiga} operators with real analytic coefficients are observed,
while in Theorem \ref{FundamentalTheoremNK} we allow the extended Gevrey regularity of the coefficient.
In particular, by the inspection of the proof, we conclude that Theorem \ref{FundamentalTheoremNK} remains valid even if $\s=1$ and $\t>1$,
that is, if the coefficients are Gevrey regular. In that sense Theorem \ref{FundamentalTheoremNK} extends
\cite[Theorem 8.6.1]{HermanderKnjiga} as well.

\par

\begin{te}
\label{FundamentalTheoremNK}
Let there be given$\t>0,$ $\s>1$, $u\in \D'(U)$ and let $\dss P(x,D)=\sum_{|\alpha|\leq m}a_{\alpha}(x)D^{\alpha}$ be partial differential operator of order $m$ such that $a_{\alpha}(x)\in \E_{\{\t,\s\}}(U)$, $|\alpha|\leq m $. Then
\be
\label{FundamentalEstimateNK}
\WF_{\{2^{\s-1}\t,\s\}}(f)\subseteq \WF_{\{2^{\s-1}\t,\s\}}(u)\subseteq \WF_{\{\t,\s\}}(f) \cup {\rm Char}(P(x,D)),
\ee
where $P(x,D)u=f$ in $\D'(U)$.
In particular,
\be
\WF_{0,\infty}(f)\subseteq \WF_{0,\infty}(u)\subseteq \WF_{0,\infty}(f) \cup {\rm Char}(P(x,D)),
\ee
where
$ \WF_{0,\infty}(u)=\bigcup_{\s>1}\bigcap_{\t>0}\WF_{\{\t,\s\}}(u)$.
\end{te}

\begin{proof}
The  pseudolocal property $ \WF_{\{2^{\s-1}\t,\s\}}(f)\subseteq \WF_{\{2^{\s-1}\t,\s\}}(u)
$ is proved in  in \cite{TT}, see also \cite{PTT-02},
so it remains to prove the second inclusion in \eqref{FundamentalEstimateNK}.

Assume that $(x_0,\xi_0)\not \in\WF_{\{\t,\s \}}(f)\cup {\rm Char}(P(x,D))$.
Then there exists a compact set $K$ containing $x_0$
and a closed cone $ \Gamma$ containing $ \xi_0 $ such that
$P_m(x,\xi)\not=0$ when $(x,\xi) \in  K \times \Gamma$ and such that
$$(K\times \Gamma)\cap \Big(\WF_{\{\t,\s \}}(f)\cup {\rm Char}(P(x,D))\Big)=\emptyset.$$
Since $K$ is fixed, the distributions $u$ and $f$ involved in the proof are of finite order denoted by the same letter $M$ for the sake of simplicity.

Let $\phi\in \D_{\{\t,\s\}}^K$ such that $\phi=1$ on some neighborhood of $x_0$.
By  Theorem \ref{TeoremaKarakterizacija} it is enough to prove that
$$
|\widehat{\phi u}(\xi)| \leq A \frac{h^{N^{\s}}N^{2^{\s-1}\t N^{\s}}}{|\xi|^N}, \;\;\; \xi\in \Gamma,\,N\in \N.
$$

\par

We divide the proof in several steps.

\textbf{Step 1.} Since $u$ is of order $M$, Paley-Wiener type estimate (see Remark \ref{PaleyRemark}) implies
$$
|\xi|^{N}|\widehat{\phi u}(\xi)|
\leq A (N^{2^{\s-1}\t N^{\s-1}})^{N}(N^{2^{\s-1}\t N^{\s-1}})^M
\leq A h^{N^{\s}}N^{2^{\s-1}\t N^{\s}}, \;\;\; N \in \N,
$$
where $A,h>0$ do not depend on $N$, and the last inequality follows from $M 2^{\s-1}\t N^{\s-1}\ln N \leq M2^{\s-1}\t N^{\s}$ after taking the exponentials.  This gives the desired estimate when $|\xi| \leq N^{2^{\s-1}\t N^{\s-1}} $, $\xi \in \Gamma$.

\par

\textbf{Step 2.}
It remains to estimate  $ |\widehat{\phi u}(\xi)|$ when $\xi \in \Gamma$,  $|\xi| > N^{2^{\s-1}\t N^{\s-1}}$ and for $N \in \N$ large enough.
We refer to  Subsection \ref{subsecrepresNK} for calculations which lead to
\be
\label{fundamentalEqualityNK}
\phi(x) =
e^{ix\cdot\xi}  P^T (x,D) \left (
\frac{e^{-i x\cdot \xi}}{P_m(x,\xi)} w_N (x, \xi) \right)
+  e_N (x, \xi), \;\;\; x\in K,\xi \in \Gamma,
\ee
where
\begin{equation}
\label{ParcijalnaSumaNK}
w_N(x,\xi)
= \sum_{k \in \mathcal{K}_1} \sum_{\mathfrak{S}_k = 0} ^{N-m}
(R_{j_1}R_{j_2}\dots R_{j_k} \phi)(x,\xi),
\end{equation}
\be
\label{ParcijalnaSumaNK1}
e_N(x,\xi)= \sum_{k \in \mathcal{K}_2 }
\sum_{\mathfrak{S}_k =N -m+1}
^{N}
(R_{j_1}R_{j_2}\dots R_{j_k} \phi)(x,\xi),
\ee
$\mathfrak{S}_k =  j_1+j_2+ \dots + j_k$, $j_i\in \{1,\dots,m\}$, $1\leq i \leq k$, and we put
\be
\label{k1}
\mathcal{K}_1 = \{k\in \N\,|\, 0\leq mk\leq N-m\},
\ee
and
\be
\label{k2}
\mathcal{K}_2 =\{k\in \N\,|\, N-m< mk\leq N\}.
\ee

The functions $R_j$ in \eqref{ParcijalnaSumaNK} and  \eqref{ParcijalnaSumaNK1} can be written as
\be \label{Rjot}
 R_j(x,\xi)= \sum_{|\alpha|\leq j}c_{\alpha,j}(x,\xi)D^{\alpha},
\ee
for suitable functions $c_{\alpha,j}(x,\xi)$ which are homogeneous of order $-j$ (with respect to $\xi$) and
such that
\be
|D^{\beta}c_{\alpha,j}(x,\xi)|\leq |\xi|^{-j} A h^{|\beta|^{\s}} |\beta|^{\t |\beta|^{\s}}, \quad \beta\in \N^d, x\in K, \xi\in \Gamma\nonumber
\ee
for some $A,h>0$ and for all $|\alpha|\leq j$, see Subsections \ref{subsecrepresNK} and \ref{coeffestimate}.

From \eqref{fundamentalEqualityNK} it follows that
\begin{multline}
\label{ocenitiJednNK}
\widehat{\phi u}(\xi)  =  \int u(x) e_N (x, \xi) e^{-i x\xi} dx +
 \int u(x) P^T (x,D) \left (
\frac{e^{-i x\cdot \xi} w_N (x, \xi)}{P_m(x,\xi)}
\right)  dx \\
=  \int u(x) e_N (x, \xi) e^{-i x\xi} dx +
 \int P (x,D)  u(x) \left (
\frac{e^{-i x\cdot \xi} w_N (x, \xi)}{P_m(x,\xi)}
\right)  dx,
\end{multline}
$x\in K,$ $\xi \in \Gamma$, and in the next steps we estimate terms on the right hand side of \eqref{ocenitiJednNK}.

\par

\begin{rem}
\label{numberofterms}
Since the  number of summands in  $w_N(x,\xi)$ and $e_N(x,\xi)$ is the same as in the case when $R_j$ have constant coefficients
we refer to \cite[Subsection 4.1]{PTT-02} where it is shown that the upper bound for the number of summands
is of the form $A\cdot C^{N}$ for suitable constants $A,C>0$.  In fact, from  \cite[Subsection 4.1]{PTT-02} it follows that the number of summands in
$$
e_N(x,\xi)= \sum_{k \in \mathcal{K}_2 }
\sum_{\mathfrak{S}_k =\lf (N/\tilde \t)^{1/\s}\rf -m+1}
^{\lf (N/\tilde \t)^{1/\s}\rf}
(R_{j_1}R_{j_2}\dots R_{j_k} \phi)(x,\xi)
$$
is bounded by $A\cdot C^{\lf (N/\tilde \t)^{1/\s}\rf}$ and the calculations remain the same after replacing $\lf (N/\tilde \t)^{1/\s}\rf$
with $N$.
\end{rem}

\textbf{Step 3.}
Note that the operators $R_j$, $1\leq j\leq m$, given in \eqref{Rjot} do not commute. For that reason we must use different arguments than those given
in \cite{PTT-02} where the operators with constant coefficients were studied.
If $M$ denotes the order of distribution $u$, then the estimates of $D^\beta (R_{j_1}...R_{j_k}\phi)$
from Subsection \ref{subsecderivativesNK} (cf. (\ref{OcenaProizvodaOpNK}))
imply
\begin{equation}
\label{OcenaPrvogSabNK}
|{\langle u(x), e_{N}(x,\xi)e^{-i x\cdot\xi}\rangle}| \leq   A \sum_{|\alpha|\leq M}|D_x^{\alpha}(e_{N}(x,\xi)e^{-ix\xi})|
\end{equation}
$$
\leq A' |\xi|^M |\xi|^{m-N} h^{N^{\s}} (N+M)^{\t (N+M)^{\s}}
= A' \frac{h^{N^{\s}} (N+M)^{\t (N+M)^{\s}}}{|\xi|^{N-m-M}},
$$
$x\in K$, $\xi \in \Gamma, $ for suitable constants $A',h>0$, and $N\in \N$ large enough.
After enumeration $N\to N+m+M$ we conclude that (\ref{OcenaPrvogSabNK}) is equivalent to
$$
|{\langle u(x), e_{N}(x,\xi)e^{-ix\cdot\xi}\rangle}|
\leq A'\frac{h^{(N+m+N)^{\s}} (N+m+2M)^{\t (N+m+2M)^{\s}}}{|\xi|^{N}}$$
$$\leq A_1\frac{h_1^{N^{\s}} N^{\t N^{\s}}}{|\xi|^{N}} ,\quad x\in K, \xi \in \Gamma,$$ where for the last inequality we used $\overline{(M.2)'}$
of the sequence $M_p^{\t,\s}=p^{\t p^{\s}}$. This is the estimate for the first term on the righthand side of (\ref{ocenitiJednNK}).

\par

\textbf{Step 4.}
To  estimate the  second term on the righthand side of
(\ref{ocenitiJednNK}) for $|\xi|>N^{2^{\s-1}\t N^{\s-1}}$, note that since $(x_0,\xi_0)\not \in \WF_{\{\t,\s\}}(f)$,
by Lemma \ref{LemaGlavnaTeorema}, there exists a compact set $\tilde K \subset \subset U$ such that $\psi\in \D_{\{\t,\s\}}(U)$,
$\psi=1$ on a neighborhood of $\tilde K$, and a conical neighborhood $V$ of $\xi_0$ such that $\Gamma \subset V$ and
\be
\label{ocenaZaf}
|{\mathcal F} (\psi f)(\eta)|\leq A\frac{h^{N^{\s}} N^{\t N^{\s}}}{|\xi|^{N}},\quad
\eta\in V, N\in \N,
\ee for some $A,h>0$. In the sequel we write $v=\psi\,f$.
Since $w_N f=w_N v$ in $\D'(U)$, we have
\begin{multline*}
\langle f(\cdot) e^{-i\xi\cdot},w_N(\cdot,\xi)/P_m(\cdot,\xi)\rangle
= {\mathcal F}_{x\to \eta}(v(x)\frac{w_N(x,\xi)}{P_m(x,\xi)})(\xi) \\
=\int_{{\bf R}^d}{\mathcal F}(v)(\xi-\eta){\mathcal F}_{x\to \eta}(\frac{w_N(x,\xi)}{P_m(x,\xi)})(\eta,\xi)\,d\eta
=I_1+I_2,
\end{multline*}
where
\be
I_1=\int_{|\eta|<\varepsilon |\xi|}{\mathcal F}(v)(\xi-\eta){\mathcal F}_{x\to \eta}(\frac{w_N(x,\xi)}{P_m(x,\xi)})(\eta,\xi)\,d\eta,\quad\nonumber
\ee
\be
I_2=\int_{|\eta|\geq\varepsilon |\xi|}{\mathcal F}(v)(\xi-\eta){\mathcal F}_{x\to \eta}(\frac{w_N(x,\xi)}{P_m(x,\xi)})(\eta,\xi)\,d\eta,\nonumber
\ee
and $0< \varepsilon <1$ is chosen so that $\xi-\eta\in V$ when $\xi\in \Gamma,$ $|\xi|>N^{2^{\s-1}\t N^{\s-1}}$,
and $|\eta|<\varepsilon |\xi|$.

\textbf{Step 5.}
Let $j_1,\dots, j_k\in \{1,\dots,m\}$ be fixed. Since the coefficients of $P_m(\cdot,\xi)$ are in
$C^{\infty}(U)$, and $P_m(x,\xi)\not=0$ when $x\in K$ and $\xi\in \Gamma$, it follows that $\dss \frac{R_{j_1}R_{j_2}\dots R_{j_k}\phi(\cdot,\xi)}{P_m(\cdot,\xi)}$ belongs to $C^{\infty} (K)$ when $\xi \in \Gamma$, and moreover it is homogeneous of order $-m-\mathfrak{S}_k$.
Hence, by Paley-Wiener type estimates it follows that there exist a constant $C>0$, such that

\begin{eqnarray*}
|{\mathcal F}_{x\to \eta}\Big(\frac{R_{j_1}R_{j_2}\dots R_{j_k}\phi(x,\xi)}{P_m(x,\xi)}\Big)(\eta,\xi)|&\leq& C|\xi|^{-m-\mathfrak{S}_k}\langle \eta \rangle^{-d-1}\\
&\leq&  C \langle \eta \rangle^{-d-1}, \quad \eta\in \Rd,\nonumber
\end{eqnarray*}
when $\xi\in \Gamma$, $|\xi|>N^{2^{\s-1}\t N^{\s-1}}$.

\par

This estimate, and the estimate for number of terms in \eqref{ParcijalnaSumaNK}
(see remark \ref{numberofterms})
imply that there exist constants $A,C>0$ such that
\begin{eqnarray}
\label{OcenaDrugogIntNK1}
\Big|{\mathcal F}_{x\to \eta}\Big(\frac{{{w_N}}(x,\xi)}{P_m(x,\xi)}\Big)\Big|&\leq&
\sum_{k \in \mathcal{K}_1}
\sum_{\mathfrak{S}_k = 0} ^{N-m}\Big|
{\mathcal F}_{x\to \eta}\Big(\frac{R_{j_1}R_{j_2}\dots R_{j_k}\phi(x,\xi)}{P_m(x,\xi)}\Big)(\eta,\xi)\Big|\nonumber\\
&\leq & A C^{N} \langle \eta \rangle^{-d-1}.
\end{eqnarray}

Since $|\eta|<\varepsilon |\xi|$ $ \Rightarrow$ $|\xi-\eta|\geq (1-\varepsilon)|\xi|$,
by using \eqref{ocenaZaf} and \eqref{OcenaDrugogIntNK1}, we obtain the desired estimate for
$I_1$:
\begin{eqnarray*}
|I_1|&\leq&  \int_{|\eta|<\varepsilon |\xi|}|{\mathcal F} (v)(\xi-\eta)|
|{\mathcal F}_{x\to \eta}(\frac{{w_N}(x,\xi)}{P_m(x,\xi)}(\eta,\xi)|\,d\eta \nonumber\\
&\leq& \int_{|\eta|<\varepsilon |\xi|} A\frac{h^{N^{\s}} N^{\t N^{\s}}}{|\xi-\eta|^{N}}
|{\mathcal F}_{x\to \eta}(\frac{{w_N}(x,\xi)}{P_m(x,\xi)}(\eta,\xi)|\,d\eta \nonumber\\
&\leq& A\frac{h^{N^{\s}} N^{\t N^{\s}}}{((1-\varepsilon)|\xi|)^{N}}
\int_{\Rd} C^{N} \langle \eta \rangle^{-d-1}\,d\eta\nonumber\\
&\leq& A_1\frac{h_1^{N^{\s}} N^{\t N^{\s}}}{|\xi|^{N}},
\;\;\;\;\;\; \xi\in  \Gamma, |\xi|>N^{2^{\s-1}\t N^{\s-1}},
\end{eqnarray*}
for some constants $A_1,h_1>0$.

\par

\textbf{Step 6.}
It remains  to estimate $I_2$. We note that $|\xi-\eta|\leq (1+1/\varepsilon)|\eta|$ when $|\eta|\geq \varepsilon |\xi|$.
Since $f$ is a distribution of order $M$, the  Paley-Wiener type estimate for $v=\psi f\in\E'(U)$ implies that $ |{\mathcal F}(v)(\eta)|\leq C \langle\eta\rangle^M$, for some constant $C>0$.
Therefore
\begin{multline}
|I_2| \leq \int_{|\eta|\geq\varepsilon |\xi|}|{\mathcal F} (v)(\xi-\eta)|
|{\mathcal F}_{x\to \eta}(\frac{{w_N}(x,\xi)}{P_m(x,\xi)})(\eta,\xi)|\,d\eta \nonumber\\
\leq \int_{|\eta|\geq \varepsilon |\xi|}\langle\xi-\eta\rangle^M
\langle\eta\rangle^{N+d+1}
\frac{|{\mathcal F}_{x\to \eta}(\frac{{w_N}(x,\xi)}{P_m(x,\xi)})(\eta,\xi)|}{\langle\eta\rangle^{N+d+1}}
 \,d\eta \nonumber\\
\leq C^{N+1}
\frac{\sup_{\eta\in \Rd}\langle \eta \rangle^{N+M+d+1}
|{\mathcal F}_{x\to \eta}(\frac{{w_N}(x,\xi)}{P_m(x,\xi)})(\eta,\xi)|}{\langle\xi\rangle^{N}},
\end{multline}
when $\xi\in  \Gamma$, $|\xi|>N^{2^{\s-1}\t N^{\s-1}}$.

To finish the proof, it remains to show that $\xi\in  \Gamma$, $|\xi|>N^{2^{\s-1}\t N^{\s-1}}$,
implies that there exist constants  $A,h>0$ such that
\be
\label{ZelimodaOcenimoNK}
\sup_{\eta \in \Rd}\langle \eta \rangle^{N+M+d+1}
|{\mathcal F}_{x\to \eta}(\frac{{w_N}(x,\xi)}{P_m(x,\xi)})(\eta,\xi)|\leq A h^{N^{\s}} N^{2^{\s-1}\t N^{\s}},
\ee
for  a sufficiently large $N\in \N$, and then we use this estimate to bound $|I_2|.$

Arguing in the similar way as in the proof of \cite[Theorem 1.1]{PTT-02}, it is sufficient to prove
\be
\label{NKocena}
\sup_{x\in K}| D^{\beta}\Big(\frac{{w_N}(x,\xi)}{P_m(x,\xi)}\Big) \leq A h^{N^{\s}} N^{2^{\s-1}\t N^{\s}},\quad \beta\in \N^d, |\beta|=N+M+d+1,
\ee
for some constants $A,h>0$, when $\xi\in  \Gamma$, $|\xi|>N^{2^{\s-1}\t N^{\s-1}}$.
Recall (see Subsection \ref{coeffestimate}),
\be
\sup_{x \in K}\Big|D^{\gamma}\frac{1}{P_m(x,\xi)}\Big|\leq |\xi|^{-m}C^{|\gamma|^{\s}+1}|\gamma|^{\t |\gamma|^{\s}}, \quad \gamma \in \N^d, \xi \in \Gamma,\nonumber
\ee
for some constant $C>0$.
Moreover, from \eqref{OcenaProizvodaOpNK} (see Subsection \ref{subsecderivativesNK}) it follows that
\begin{multline}
\sup_{x \in K}|D^{\gamma}w_N(x,\xi)|\leq Ah ^{N^{\s}}
\sum_{k \in \mathcal{K}_1}
\sum_{\mathfrak{S}_k = 0} ^{N-m}|\xi|^{-\mathfrak{S}_k} (\mathfrak{S}_k+|\gamma|)^{\t (\mathfrak{S}_k+|\gamma|)^{\s}},\nonumber
\end{multline} for some constants $A,h>0$, when $\xi \in \Gamma$.

Hence, for $x\in K$ and $\xi\in \Gamma$, $|\xi|>N^{2^{\s-1}\t N^{\s-1}}$, we obtain
\begin{multline}
\label{NKracun}
\left |D^{\beta}\frac{{w_N}(x,\xi)}{P_m(x,\xi)} \right |\leq
\sum_{k \in \mathcal{K}_1}
\sum_{\gamma\leq \beta}{\beta\choose \gamma}|D^{\beta-\gamma}\frac{1}{P_m(x,\xi)}||D^{\gamma}w_N(x,\xi)|\\
\leq Ah ^{N^{\s}}\sum_{\gamma\leq \beta}\sum_{k \in \mathcal{K}_1}\sum_{\mathfrak{S}_k = 0} ^{N-m} |\xi|^{-\mathfrak{S}_k-m}
{\beta\choose \gamma} C^{|\beta-\gamma|^{\s}+1}|\beta-\gamma|^{\t |\beta-\gamma|^{\s}}
(\mathfrak{S}_k+|\gamma|)^{\t (\mathfrak{S}_k+|\gamma|)^{\s}}\\
\leq A'h ^{'N^{\s}}\sum_{\gamma\leq \beta} \sum_{k \in \mathcal{K}_1}\sum_{\mathfrak{S}_k = 0} ^{N-m} {\beta\choose \gamma} |\xi|^{-\mathfrak{S}_k}(\mathfrak{S}_k+|\beta|)^{\t (\mathfrak{S}_k+|\beta|)^{\s}},
\end{multline}
for $\beta \in \N^d$, $|\beta|=N+M+d+1$, where we used $(M.1)$ property of the sequence $M_p^{\t,\s}=p^{\t p^{\s}}$.

Since $\mathfrak{S}_k\leq N-m$ it follows that $N>\mathfrak{S}_k$ and therefore
$$|\xi|>N^{2^{\s-1}\t N^{\s-1}}>{\mathfrak{S}_k}^{2^{\s-1}\t {\mathfrak{S}_k}^{\s-1}}.$$
Now $\overline{(M.2)}$ property of $M_p^{\t,\s}=p^{\t p^{\s}}$ implies

\begin{multline}
\label{NKocena1}
 |\xi|^{-\mathfrak{S}_k}(\mathfrak{S}_k+|\beta|)^{\t (\mathfrak{S}_k+|\beta|)^{\s}}\leq \frac{(\mathfrak{S}_k+|\beta|)^{\t (\mathfrak{S}_k+|\beta|)^{\s}}}{{\mathfrak{S}_k}^{2^{\s-1}\t {\mathfrak{S}_k}^{\s}}}\\
 \leq C^{{\mathfrak{S}_k}^{\s}+|\beta|^{\s}}\frac{{\mathfrak{S}_k}^{2^{\s-1}\t {\mathfrak{S}_k}^{\s}}{|\beta|^{2^{\s-1}\t |\beta|^{\s}}}}{{\mathfrak{S}_k}^{2^{\s-1}\t {\mathfrak{S}_k}^{\s}}}\\
 = C^{{\mathfrak{S}_k}^{\s}+|\beta|^{\s}} (N+M+d+1)^{2^{\s-1}\t (N+M+d+1)^{\s}}\leq C_1^{ N^{\s}} N^{2^{\s-1}\t N^{\s}},
\end{multline}
for some constant $C_1>0$ where the last inequality follows from $\overline{(M.2)'}$ property of $M_p^{\t,\s}$. Using the estimate for number of terms in $w_N$, by \eqref{NKracun} and \eqref{NKocena1}, the estimate \eqref{NKocena} follows.

By the similar arguments as in the proof of  \cite[Theorem 1.1]{PTT-02}, (\ref{ZelimodaOcenimoNK}) follows from \eqref{NKocena}  since
$$\pi_1(\supp\frac{{w_N}(x,\xi)}{P_m(x,\xi)})\subseteq K.$$

Therefore
\be
\label{poslednjiUslov1NK}
|I_2|\leq A \frac{h^{N^{\s}} N^{2^{\s-1}\t N^{\s}}}{|\xi|^{N}}\,
\ee
for suitable constants $A,h>0$ and $N$ sufficiently large, and the theorem is proved.
\end{proof}

\subsection{Representing $\widehat{\phi u} (\xi)$ by an approximate solution} \label{subsecrepresNK}
In this subsection we derive \eqref{fundamentalEqualityNK}, \eqref{ParcijalnaSumaNK} and \eqref{ParcijalnaSumaNK1}.

Let $P^T (x,D)=\dss\sum_{|\alpha|\leq m} b_{\alpha}(x) D^{\alpha}$,
$b_{\alpha}(x)\in \E_{\{\t,\s\}}(U)$ be the transpose of $P(x,D)$.
If $v(x,\xi)$ is the solution of the equation
\begin{equation}
\label{ResitiJednacinuNK}
e^{i x\xi} P^T (x,D) v(x, \xi)
= \phi(x), \quad x\in K, \xi\in \Gamma,
\end{equation}
then
$$
\widehat{\phi u}(\xi)  =  \int u(x) \phi(x)e^{-i x\xi} dx
=  \int u(x) P^T (x,D) v(x,\xi) dx, \;\;\;\xi \in \Gamma.
$$

Similarly as in \cite{HermanderKnjiga} and \cite{PTT-01} we may assume that
$ \displaystyle v (x, \xi) = \frac{e^{-i x\xi}w(x, \xi)}{P_m(x,\xi)}, $
for some $ w (\cdot,\xi) \in C^{\infty}(K),$
so that the left hand side of
\eqref{ResitiJednacinuNK} becomes
\begin{multline}
e^{i x\xi} P^T (x,D)(\frac{w(x,\xi)e^{-i x\xi}}{{P_m(x,\xi)}}) \nonumber \\
= e^{i x\xi} \sum_{|\alpha|\leq m}\sum_{\beta\leq\alpha} {\alpha \choose \beta}b_{\alpha}(x) D^{\alpha-\beta}(e^{-i x\xi}) D^{\beta}\Big(\frac{w(x,\xi)}{P_m(x,\xi)}\Big)   \nonumber \\
= \sum_{|\alpha|\leq m} \sum_{\beta\leq\alpha}\sum_{\gamma\leq\beta} {\alpha \choose \beta}{\beta \choose \gamma}b_{\alpha}(x)(-\xi)^{\alpha-\beta}
D^{\gamma}\Big(\frac{1}{P_m(x,\xi)}\Big) D^{\beta-\gamma}w(x,\xi),
\end{multline}
\be \label{HermanderRacunNK}
= ( I -R(x,\xi)) w(x,\xi), \;\;\;
x\in K, \xi\in \Gamma,
\ee
where
\be
\label{ReprezentacijaRJ}
R(x,\xi)=\sum_{j=1}^m R_j(x,\xi),\quad \dss R_j(x,\xi)= \sum_{|\alpha|\leq j}c_{\alpha,j}(x,\xi)D^{\alpha},\nonumber
\ee
for suitable functions $c_{\alpha,j}(x,\xi)$ which are homogeneous of order $-j$ and
\be
\label{OcenaKoeficijentiRj}
|D^{\beta}c_{\alpha,j}(x,\xi)|\leq |\xi|^{-j} A h^{|\beta|^{\s}} |\beta|^{\t |\beta|^{\s}}, \quad \beta\in \N^d, x\in K, \xi\in \Gamma
\ee
for some $A,h>0$ and for all $|\alpha|\leq j$. We refer to Subsection \ref{coeffestimate} for the calculus which shows how
\eqref{HermanderRacunNK} implies \eqref{OcenaKoeficijentiRj}.

Therefore \eqref{ResitiJednacinuNK} can be rewritten in the following convenient form:
\be
\label{JednacinaKojuResavamoNK}
(I -R(x,\xi))w(x,\xi)= \phi(x)\,
\quad x\in K, \xi \in \Gamma.
\ee
which gives rise to approximate solutions as follows.

\par

Note that the order of operator $R^k$, $k\in \N$, is $mk$. We compute
$$
\sum_{k \in \mathcal{K}_1}R^k- R \sum_{k \in \mathcal{K}_1}R^{k} = \sum_{k \in \mathcal{K}_1}R^k-\sum_{k \in \mathcal{K}_1}R^{k+1}
$$
\be
\label{RacunOperatoriNK}
= \sum_{k \in \mathcal{K}_1}R^k-\sum_{\{k\in \N\,|\, m\leq mk\leq N\}}R^k =I-\sum_{k \in \mathcal{K}_2}
R^k
\ee
where $\mathcal{K}_1$ is given by \eqref{k1} and in the last equality we used
$$
\mathcal{K}_1
\cap\{k\in \N\,|\, m\leq mk\leq N\}
=\{k\in \N\,|\, m\leq mk\leq N-m\}.
$$
Moreover, since the operators $R_j$, $1\leq j\leq m$, do not commute we can write
$$
\sum_{k \in \mathcal{K}_1} R^k= \sum_{k \in \mathcal{K}_1}\sum_{\mathfrak{S}_k =0} ^{N-m}
R_{j_1}R_{j_2}\dots R_{j_k},
$$
and
$$
\sum_{k \in \mathcal{K}_2} R^k= \sum_{k \in \mathcal{K}_2}
\sum_{\mathfrak{S}_k =N-m+1} ^{N}
R_{j_1}R_{j_2}\dots R_{j_k}
$$
where $\mathfrak{S}_k = j_1+j_2+\dots+j_k$, $j_i\in \{1,\dots,m\}$, $1\leq i\leq k$.

Now,
\begin{multline}
(I -R(x,\xi)) (\sum_{k \in \mathcal{K}_1}R^k \phi (x))
= (\sum_{k \in \mathcal{K}_1}R^k- R \sum_{k \in \mathcal{K}_1}R^{k})\phi (x)
\nonumber \\
= (I-\sum_{k \in \mathcal{K}_2}R^{k})\phi (x) = \phi (x) - \sum_{k \in \mathcal{K}_2}R^{k} \phi (x),
\end{multline}
and if we put $ w_N = \sum_{k \in \mathcal{K}_1}R^k \phi$ and $ e_N = \sum_{k \in \mathcal{K}_2}\phi$
we conclude that
$$
(I-R)w_N(x,\xi)=\phi(x)-e_N(x,\xi),\quad N\in \N, x\in K, \xi\in \Gamma,
$$
with $w_N$ and $e_N$ given by \eqref{ParcijalnaSumaNK} and \eqref{ParcijalnaSumaNK1} respectively.


\subsection{Estimates for $c_{\alpha,j}(x,\xi)$} \label{coeffestimate}
In this subsection we show that \eqref{HermanderRacunNK} implies \eqref{OcenaKoeficijentiRj}.
An essential argument in
this part of the proof is the inverse-closedness property presented in Theorem \ref{teoremaInverseClosedness}.

Recall,
\be
\label{OcenaPM1}
D^{\alpha}\Big(\frac{1}{P_m(x,\xi)}\Big)=\alpha! \sum_{(s,p,j)\in \pi} \frac{(-1)^j j!}{(P_m(x,\xi))^{j+1}}\prod_{k=1}^{s}\frac{1}{j_k!}\Big(\frac{1}{p_k !}D^{p_k} P_m(x,\xi) \Big)^{j_k},
\ee
for $\alpha\in \N^d$, where sum is taken over all decompositions $(s,p,j)$ of the form
\be
\label{oblikParticijaFaaGlavnaTeorema}
\alpha=j_1 p_1+j_2 p_2+\dots+j_s p_s,\nonumber
\ee with $\dss j=\sum_{i=1}^{s} j_i\in\{0,1,\dots, |\alpha|\}$, $p_i\in \N^d$, $|p_i|\in\{1,\dots,|\alpha|\}$ for $i\in \{1,\dots,s\}$, $s\leq |\alpha|$. (see Subsection \ref{subsecInverse})

Since the coefficients of $P_m(x,\xi)$ belong to $\E_{\{\t,\s\}}(U)$ it follows
that
\be
\label{OcenaPM2}
\sup_{x\in K}|D^{p_k}P_m(x,\xi)|\leq A h^{|p_k|^{\s}}|p_k|^{\t |p_k|^{\s}}|\xi|^m,
\ee
for some $A,h>0$.
Moreover, from $(K\times \Gamma) \cap {\rm Char}(P)=\emptyset$ it follows that
\be
\label{OcenaPM3}
\sup_{x\in K}|P_m(x,\xi)|\geq C'|\xi|^{m}.
\ee

Hence, by using \eqref{OcenaPM1}, \eqref{OcenaPM2} and \eqref{OcenaPM3}  we obtain
\begin{eqnarray*}
\label{InverseClosed}
|D^{\alpha}\Big(\frac{1}{P_m(x,\xi)}\Big)|&\leq & |\alpha|!\sum_{(s,p,j)\in \pi}\frac{j!}{j_1!\dots j_s! |P_m(x,\xi)|^{j+1}} \\
&\times & \prod_{k=1}^{s}\Big(\frac{1}{p_k!}|D^{p_k} P_m(x,\xi)| \Big)^{j_k}\\
&\leq & |\alpha|!\sum_{(s,p,j)\in \pi}\frac{|\xi|^{m j} j!}{|\xi|^{m(j+1)}j_1!\dots j_s! |P_m(x,\xi)|^{j+1}}\\
&\times & \prod_{k=1}^{s}\Big(\frac{1}{p_k!}A h^{|p_k|^{\s}}|p_k|^{\t |p_k|^{\s}} \Big)^{j_k}\\
&\leq &|\xi|^{-m} A' h^{'|\alpha|^{\s}+1} |\alpha|^{\t |\alpha|^{\s}},
\end{eqnarray*}
for some $A,A'  $, $h,h'>0$,
where the last inequality follows by calculation from the proof of Theorem \ref{teoremaInverseClosedness}.

In particular, we have proved that $\dss \frac{1}{P_m(\cdot,\xi)}\in \E_{\{\t,\s,h\}}(K)$ for some $h>0$ and for every $\xi\in \Gamma$. From the algebra property of extended Gevrey classes it follows that $\dss b_{\alpha}(\cdot)\partial^{\gamma}\frac{1}{P_m(\cdot,\xi)}\in \E_{\{\t,\s,h'\}}(K)$ for some $h'>0$, where $|\gamma|\leq|\alpha|\leq m$ and $b_{\alpha}(x)$ are the coefficients of $P^T(x, D)$.

These estimates, together with  \eqref{HermanderRacunNK} give  \eqref{OcenaKoeficijentiRj}.

\subsection{Estimates for $ D^\beta ( R_{j_1}...R_{j_k}\phi)$} \label{subsecderivativesNK}
In this subsection we  follow the idea presented in \cite[Lemmas 8.6.2 and 8.6.3]{HermanderKnjiga}.
As in Subsection \ref{subsecrepresNK} we put
$$
\mathfrak{S}_k=j_1+\dots+j_k, \quad N-m\leq \mathfrak{S}_k\leq N,
$$
for $k\in \N$ such that $mk \leq  N$, and let $|\beta|\leq M$ where $M$ is order of distribution $u$.

Recall, $ \displaystyle R_j(x,\xi)= \sum_{|\alpha|\leq j}c_{\alpha,j}(x,\xi)D^{\alpha},$
and note that by successive applications of the Leibniz rule $D^{\beta} (R_{j_1}...R_{j_k}\phi)$ can be written as a sum of terms of the form
$$
(D^{\gamma_0}c_{\alpha_{j_1},j_1}(x,\xi))(D^{\gamma_1}c_{\alpha_{j_2},j_2}(x,\xi))\dots(D^{\gamma_{k-1}}c_{\alpha_{j_k},j_k}(x,\xi))(D^{\gamma_k}\phi(x)).
$$
Put $a_i = |\gamma_i|$ so that
\be
\label{Suma1}
a_0+\dots+a_k= \mathfrak{S}_k+|\beta|,
\ee
\be
\label{Suma2}
a_0\leq |\beta|,
\ee
and
\be
\label{Suma3}
a_i\leq \sum_{t=1}^i j_t +|\beta|, \quad 1\leq i\leq k.
\ee
From \eqref{OcenaKoeficijentiRj} it follows that
\be
|D^{\gamma_{i-1}}c_{\alpha_{j_i},j_i}(x,\xi)|\leq |\xi|^{-j_i} A h^{a_{i-1}^{\s}} a_{i-1}^{\t a_{i-1}^{\s}}, \quad \gamma_{i-1}\in \N^d, x\in K, \xi\in \Gamma,\nonumber
\ee
for some constants $A,h>0$ and for all $|\alpha_{j_i}|\leq j_i$, $i=1,\dots, k$.

Observe that the number of multiindices $\gamma_0,\dots,\gamma_k$ with the property \eqref{Suma1} is $\dss {\mathfrak{S}_k+|\beta|\choose a_0,\dots, a_k}$. In the sequel we write $\dss \sum$ when the sum is taken over all multiindices $\gamma_0,\dots,\gamma_k$ which satisfies \eqref{Suma1}-\eqref{Suma3}.

Since $\phi \in \D_{\{\t,\s\}}^K$, for $x\in K$ and $\xi \in \Gamma$, we estimate
\begin{multline}
|(D^\beta R_{j_1}...R_{j_k}\phi)(x,\xi)|\leq\\ \sum{\mathfrak{S}_k+|\beta|\choose a_0,\dots, a_k}\Big(\prod_{i=1}^k|D^{\gamma_{i-1}}c_{\alpha_{j_i},j_i}(x,\xi)|\Big)\cdot |D^{\gamma_k}\phi (x)|\nonumber\\
\leq |\xi|^{-\mathfrak{S}_k} \sum{\mathfrak{S}_k+|\beta|\choose a_0,\dots, a_k}\Big(\prod_{i=1}^k A h^{a_{i-1}^{\s}} a_{i-1}^{\t a_{i-1}^{\s}}\Big)\cdot \Big(A h^{{a_k}^{\s}}{a_k}^{\t {a_k}^{\s}}\Big)  \nonumber\\
\leq |\xi|^{m-N} A^{\frac{N}{m} +1}  h^{N^{\s}} \\ \sum{\mathfrak{S}_k+|\beta|\choose a_0,\dots, a_k}\Big(\prod_{i=1}^{k+1} a_{i-1}^{\t a_{i-1}^{\s}}\Big)\\
\leq |\xi|^{m-N} A' h^{' N^{\s}} \sum{\mathfrak{S}_k+|\beta|\choose a_0,\dots, a_k}\Big(\prod_{i=1}^{k+1} a_{i-1}^{\t a_{i-1}^{\s}}\Big),
\end{multline}
for some $A',h'>0$. By the almost increasing property of $M_p^{\t,\s}=p^{\t p^{\s}}$ it follows that
$$
\prod_{i=1}^{k+1} a_{i-1}^{\t a_{i-1}^{\s}}\leq C^{a_0+\dots+a_{k}}\frac{a_0!\cdot\cdot\cdot a_{k}!}{(a_0+\dots+a_{k})!}(a_0+\dots+a_{k})^{\t (a_0+\dots+a_{k})^{\s}}$$
$$
=C^{\mathfrak{S}_k+|\beta|} \frac{a_0!\cdot\cdot\cdot a_{k}!}{(\mathfrak{S}_k+|\beta|)!}(\mathfrak{S}_k+|\beta|)^{\t (\mathfrak{S}_k+|\beta|)^{\s}},
$$ for some $C>0$, wherefrom
\begin{multline}
\sum{\mathfrak{S}_k+|\beta|\choose a_0,\dots, a_k}\Big(\prod_{i=1}^{k+1} a_{i-1}^{\t a_{i-1}^{\s}}\Big)
\leq \sum
\frac{a_0!\cdot\cdot\cdot a_{k}!}{(\mathfrak{S}_k+|\beta|)!}\nonumber \\
\cdot C^{\mathfrak{S}_k+|\beta|}\frac{(\mathfrak{S}_k+|\beta|)!}{a_0!\dots  a_k!}
(\mathfrak{S}_k+|\beta|)^{\t (\mathfrak{S}_k+|\beta|)^{\s}}\nonumber\\
=C^{\mathfrak{S}_k+|\beta|}(\mathfrak{S}_k+|\beta|)^{\t (\mathfrak{S}_k+|\beta|)^{\s}}\sum_{a_0+\dots+a_k=\mathfrak{S}_k+|\beta|} \,1.\nonumber\\
\leq C^{N}(N+M)^{\t (N+M)^{\s}}{\mathfrak{S}_k+|\beta|-1\choose k }\leq C^{'N}(N+M)^{\t (N+M)^{\s}}
\end{multline} for suitable $C'>0$.

Hence we conclude that there exist  constants $A,h>0$ such that
\begin{equation}
\label{OcenaProizvodaOpNK}
|(D^\beta R_{j_1}...R_{j_k}\phi)(x,\xi)|
\leq A |\xi|^{m-N} h^{N^{\s}} (N+M)^{\t (N+M)^{\s}},
\end{equation}
$x\in K,$ $\xi\in \Gamma,$ which gives the desired estimate.

\end{document}